\documentclass[reqno]{amsart}%
\usepackage{amstext}
\usepackage{amsfonts}
\usepackage{amsmath}
\usepackage{amssymb}
\usepackage{hyperref}
\usepackage{graphicx}%
\numberwithin{equation}{section}
\newtheorem{theorem}{Theorem}[section]

\newtheorem{proposition}[theorem]{Proposition}

\newtheorem{lemma}[theorem]{Lemma}

\begin{document}
\title[3D Navier Stokes Equations.]{Hausdorff measure of Vorticity Nodal Sets for the 3D Hyperviscous Navier
Stokes Equations with General forces}
\author{Abdelhafid Younsi}
\address{Department of Mathematics and Computer Science, University of Djelfa, Algeria.}
\email{younsihafid@gmail.com}
\subjclass[2000]{ 35K55, 35Q30, 76D05, 76F20, 76F70}
\keywords{Navier-Stokes equations,\thinspace Hyperviscosity, Vorticity, Turbulence}

\begin{abstract}
In this paper, we modified the three dimensional Navier-Stokes equations by
adding a $l$-Laplacian. We provide upper bounds on the two-dimensional
Hausdorff measure $\mathcal{H}_{l}^{2}$ of $N_{\omega}^{0}=\{x\in\Omega\subset%
\mathbb{R}
^{3}/$ $\omega(x,t)=0\}$ the level sets of the vorticity $\omega$\ of
solutions. We express them in terms of the Kolmogorov length-scale and the
Landau--Lifschitz estimates of the number of degrees of freedom in turbulent
flow. We also, under certain hypothesis recover the two-dimensional Hausdorff
measure estimates for the usual 3D Navier--Stokes equations with potential
force. Moreover, we show that the estimates depend on $l$, this result
suggests that the modified Navier Stokes system is successful model of
turbulence and the size of the nodal set $\mathcal{H}_{l}^{2}(N_{\omega}^{0})$
leads the way for developing the turbulence theory.

\end{abstract}
\maketitle

\section{Introduction}

In this paper, we provide upper bounds on the two-dimensional Hausdorff
measure $\mathcal{H}_{l}^{2}$ of $N_{\omega}^{0}$ the level sets associated
with the vorticity of modified three dimensional Navier-Stokes equations. We
modified the 3D Navier-Stokes system by adding a higher-order viscosity term
to the conventional system%
\begin{equation}%
\begin{array}
[c]{c}%
\dfrac{du}{dt}+\varepsilon\left(  -\triangle\right)  ^{l}u-\nu\triangle
u+\left(  u.\nabla\right)  u+\nabla p=f\left(  x\right)  ,\text{ in }%
\Omega\times\left(  0,\infty\right) \\
\text{div}u=0,\text{ in }\Omega\times\left(  0,\infty\right)  \text{,}\\
p(x+Le_{i},t)=p(x,t),\text{ }u(x+Le_{i},t)=u(x,t)\text{\ }i=1,...,d\text{
\ }t\in\left(  0,\infty\right) \\
u\left(  x,0\right)  =u_{0}\left(  x\right)  ,\text{in }\Omega\text{,}%
\end{array}
\label{1}%
\end{equation}
on $\Omega=\left(  0,L\right)  ^{d}$ with periodic boundary conditions and
$\left(  e_{1},...,e_{d}\right)  $ is the natural basis of $%
\mathbb{R}
^{d}$. Here $\varepsilon>0$ is the artificial dissipation parameter and
$\nu>0$ is the kinematic viscosity of the fluid, $l>1$. The functions $u\ $is
the velocity vector field, $p$ is the pressure, and $f$ is a given force
field. For $\varepsilon=0$, the model is reduced to the Navier-Stokes system.

In the work \cite{23}, the strong convergence of the solution of this problem
to the solution of the conventional system as the regularization parameter
goes to zero, was established for each dimension $d\leq4$.

Mathematical model for such fluid motion has been used extensively in
turbulence simulations (see e.g. \cite{6}) also see Borue and Orsag \cite{1a,
1b}. For further discussion of theoretical results concerning (%
\ref{1}
), see \cite{0b, 0a,18a, 23}.

For the 3D Navier--Stokes system weak solutions of problem are known to exist
by a basic result by J. Leray from 1934 \cite{17b}, only the uniqueness of
weak solutions remains as an open problem. Then the known theory of global
attractors of infinite dimensional dynamical systems is not applicable to the
3D Navier--Stokes system.

In particular, in case one accepts the point of view that the dimension of a
global attractor for the Navier--Stokes equations is associated with the
number of degrees of freedom in turbulent flows, then the two-dimensional
Hausdorff measure $\mathcal{H}_{l}^{2}(N_{\omega}^{0})$ is an important way to
the understanding of turbulence theory \cite{23}.

We are interested in the three dimensional case. Let $P_{m}$ be the projection
onto the first $m$ eigenspaces of the Stokes operator $A=-\triangle$ and let
$N_{\omega}^{0}=\{x\in\Omega\subset%
\mathbb{R}
^{3}/$ $\omega(x,t)=0\}$ the nodal sets of the vorticity $\omega$ for
solutions of the equation (%
\ref{1}
). We provide an upper bound on the size of the nodal sets $\mathcal{H}%
_{l}^{2}(N_{\omega}^{0})$ and we show that, the bounds necessarily depend on
$m$ and $l\ $this dependence is a fractional power of $l$. Thus answering a
question raised by J. Avrin \cite{0b}. We also obtain here scale-invariant
estimates on the two-dimensional Hausdorff measure $\mathcal{H}_{l}%
^{2}(N_{\omega}^{0})$ in terms of the Landau--Lifschitz theory of the number
of degrees of freedom in turbulent flow. Since expressing the above estimates
in terms of the (dimensionless) Grashoff number $G$. In order to obtain an
upper bound on the Hausdorff measure of level sets associated, we use the
method from \cite{16} (see also \cite{4}, \cite{5}).

The main purpose of the present article is to study the dependence of the
two-dimensional Hausdorff measure $\mathcal{H}_{l}^{2}(N_{\omega}^{0})$ on the
parameter $l$. Using a family of Kolmogorov flows as base flows we can deduce
also upper bounds on the Hausdorff measure $\mathcal{H}_{l}^{2}(N_{\omega}%
^{0})$. We also find here that the upper bounds on the two-dimensional
Hausdorff measure of $N_{\omega}^{0}$ converges to the corresponding upper
bounds on $\mathcal{H}_{1}^{2}(N_{\omega}^{0})$ the two-dimensional Hausdorff
measure of the nodal sets of the usual 3D Navier-Stokes as $l=1$. Under
certain hypothesis we recover the two-dimensional Hausdorff measure
$\mathcal{H}_{1}^{2}(N_{\omega}^{0})$ estimates for the usual 3D
Navier--Stokes equations with potential force. We extend the method from
\cite{15a} to a 3D Navier-Stokes with general forcing modified by
$l$-Laplacian. These estimates are obtained without using the Dirichlet
quotients \cite{15a}.

We note, however, that for the incompressible 3D Navier-Stokes equations with
general force, it seems not so easy to get some better estimates on the
Hausdorff measure of the level sets associated with the vorticity as in the
case of potential force studied in \cite{15a,16a} for periodic solutions of
the 2D. Related results for the 3D Navier--Stokes equations (with general
forcing) can be found in \cite{1}. The upper bounds on the Hausdorff measures
of the level sets associated with solutions of some other partial differential
equations were obtained in \cite{5},\cite{12}, \cite{13}, \cite{14},
\cite{16}, \cite{17}, and \cite{18}.

The paper is organized as follows. In Section 2, we present the relevant
mathematical framework for the paper. In Section 3, we provide upper bounds
for the two-dimensional Hausdorff measure $\mathcal{H}_{l}^{2}$ of the level
sets associated with the vorticity of the Navier-Stokes system with hyperdissipation.

\section{Notations and preliminaries}

In this section we introduce notations and the definitions of standard
functional spaces that will be used throughout the paper. We denote by
$H^{m}\left(  \Omega\right)  $, the Sobolev space of $L_{\text{per}}$ periodic
functions. These spaces are endowed with the inner product%
\[
\left(  u,v\right)  =%
{\textstyle\sum\limits_{\left\vert \beta\right\vert \leq m}}
(D^{\beta}u,D^{\beta}v)_{L^{2}\left(  \Omega\right)  }\text{ and the norm
}\left\Vert u\right\Vert _{m}=%
{\textstyle\sum\limits_{\left\vert \beta\right\vert \leq m}}
(\left\Vert D^{\beta}u\right\Vert _{L^{2}\left(  \Omega\right)  }^{2}%
)^{\frac{1}{2}}.
\]
Each $u\in L_{per}$ can be identified with its Fourier expansion%
\[
u\left(  x\right)  =%
{\displaystyle\sum\limits_{k\in\mathbb{Z}^{3}}}
u_{k}\exp(2i\pi k.\frac{x}{L})
\]
where $u_{k}$ $\in$ $%
\mathbb{C}
^{3}$ satisfy $\overline{u}_{k}=u_{-k}$. Then $u$ is in $L^{2}$ if and only
if
\[
\left\Vert u\right\Vert _{L^{2}}^{2}=\left\vert \Omega\right\vert
{\displaystyle\sum\limits_{k\in\mathbb{Z}^{3}}}
\left\vert u_{k}\right\vert ^{2}<\infty,\text{ \ \ \ }\left\vert
\Omega\right\vert =L^{3},
\]
then the Sobolev space $u\in H^{m}\left(  \Omega\right)  $, $m\in%
\mathbb{R}
^{+}$ can be characterized by%
\[
H^{m}\left(  \Omega\right)  =\{u,\overline{u}_{k}=u_{-k},%
{\displaystyle\sum\limits_{k\in\mathbb{Z}^{3}}}
k^{2m}\left\vert u_{k}\right\vert ^{2}<\infty.\}
\]
$H^{-m}\left(  \Omega\right)  $ denote the dual space of $H^{m}\left(
\Omega\right)  $.

We denote by $\dot{H}^{m}\left(  \Omega\right)  $ the subspace of
$H^{m}\left(  \Omega\right)  $ with, zero average%

\[
\dot{H}^{m}\left(  \Omega\right)  =\{u\in H^{m}\left(  \Omega\right)  ;\text{
}%
{\displaystyle\int\limits_{\Omega}}
u\left(  x\right)  dx=0\}.
\]
For $m=0$, we have $\dot{H}^{m}\left(  \Omega\right)  =\dot{L}^{2}\left(
\Omega\right)  $.

\begin{itemize}
\item We introduce the following solenoidal subspaces $V_{s},$ $s\in%
\mathbb{R}
^{+}$ which are important to our analysis%
\begin{align*}
V_{0}\left(  \Omega\right)   &  =\{u\in\dot{L}^{2}\left(  \Omega\right)
,\text{div}u=0,u.n\mid_{\Sigma_{i}}=-u.n\mid_{\Sigma_{i+3}},i=1,2,3\};\\
V_{1}\left(  \Omega\right)   &  =\{u\in\dot{H}^{1}\left(  \Omega\right)
,\text{div}u=0,\gamma_{0}u\mid_{\Sigma_{i}}=\gamma_{0}u\mid_{\Sigma_{i+3}%
},i=1,2,3\}.\\
V_{2}\left(  \Omega\right)   &  =\{u\in\dot{H}^{2}\left(  \Omega\right)
,\text{div}u=0,\gamma_{0}u\mid_{\Sigma_{i}}=\gamma_{0}u\mid_{\Sigma_{i+3}%
},\gamma_{1}u\mid_{\Sigma_{i}}=-\gamma_{1}u\mid_{\Sigma_{i+3}},i=1,2,3\},
\end{align*}

\end{itemize}

see \cite[Chapter III, Section 2]{21}. We refer the reader to R.Temam
\cite{22} for details on these spaces. Here the faces of $\Omega$ are numbered
as%
\[
\Sigma_{i}=\partial\Omega\cap\left\{  x_{i}=0\right\}  \text{ and }%
\Sigma_{i+3}=\partial\Omega\cap\left\{  x_{i}=L\right\}  ,\text{ }i=1,2,3.
\]
Here $\gamma_{0}$, $\gamma_{1}$ are the trace operators and $n$ is the unit
outward normal on $\partial\Omega$.

\begin{itemize}
\item The space $V_{0}$ is endowed with the inner product $\left(  u,v\right)
_{L^{2}\left(  \Omega\right)  }$ and norm $\left\Vert u\right\Vert
=\nolinebreak\left(  u,u\right)  _{L^{2}\left(  \Omega\right)  }^{1\prime2}$.

\item $V_{1}$ is the Hilbert space with the norm $\left\Vert u\right\Vert
_{1}=\left\Vert u\right\Vert _{V_{1}}$. The norm induced by $\dot{H}%
^{1}\left(  \Omega\right)  $\ and the norm $\left\Vert \nabla u\right\Vert
$\ are equivalent in $V_{1}$.

\item $V_{2}$ is the Hilbert space with the norm $\left\Vert u\right\Vert
_{2}=\left\Vert u\right\Vert _{V_{2}}$. In $V_{2}$ the norm induced by
$\dot{H}^{2}\left(  \Omega\right)  $ is equivalent to the norm $\left\Vert
\triangle u\right\Vert $.
\end{itemize}

$V_{s}^{\prime}$ denote the dual space of $V_{s}$.

Let $P$ be the orthogonal projection in $L_{\text{per}}^{2}\left(
\mathbb{R}
^{3}\right)  ^{3}$ with the range $H$.\newline Let $A=-P\triangle$ the Stokes
operator. It is easy to check that $Au=-\triangle u$ for every $u\in D\left(
A\right)  $. We recall that the operator $A$ is a closed positive self-adjoint
unbounded operator, with $D\left(  A\right)  =\left\{  u\in V_{0}\text{,
}Au\in V_{0}\right\}  $. We have in fact,
\[
D\left(  A\right)  =\dot{H}^{2}\left(  \Omega\right)  \cap V_{0}=V_{2}\text{.}%
\]
The eigenvalues of $A$ are $\left\{  \lambda_{j}\right\}  _{j=1}^{j=\infty}$,
$0$ $<$ $\lambda_{1}\leq\lambda_{2}\leq...$and the corresponding orthonormal
set of eigenfunctions $\left\{  w_{j}\right\}  _{j=1}^{j=\infty}$ is complete
in $V_{0}$%
\[
Aw_{j}=\lambda_{j}w_{j},\ \ \ w_{j}\in D(A),\forall j.
\]
The spectral theory of $A$ allows us to define the powers $A^{l}$ of $A$ for
$l\geq1$, $A^{l}$ is an unbounded self-adjoint operator in $V_{0}$ with a
domain $D(A^{l})$ dense in $V_{2}\subset V_{0}$. We set here%
\[
A^{l}u=\left(  -\triangle\right)  ^{l}u\text{\ for }u\in D\left(
A^{l}\right)  =V_{2l}\cap V_{0}\text{.}%
\]
The space $D\left(  A^{l}\right)  $ is endowed with the scalar product and the
norm%
\begin{equation}
\left(  u,v\right)  _{D(A^{l})}=\left(  A^{l}u,A^{l}v\right)  \text{,
}\left\Vert u\right\Vert _{D\left(  A^{l}\right)  }=\{\left(  u,v\right)
_{D\left(  A^{l}\right)  }\}^{\frac{1}{2}}. \label{a1}%
\end{equation}
In the case for $l>0,$ we have $D\left(  A^{l}\right)  =\{u\in H,%
{\displaystyle\sum\limits_{j=1}^{\infty}}
\lambda_{j}^{2l}(u,w_{j})^{2}<\infty\}$. For $l\in%
\mathbb{R}
$ the scalar product and the norm in (%
\ref{a1}
) can wiriten alterntivly as
\begin{equation}
\left(  u,v\right)  _{D(A^{l})}=%
{\displaystyle\sum\limits_{j=1}^{\infty}}
\lambda_{j}^{2l}(u,w_{j})(v,w_{j}),\left\Vert u\right\Vert _{D\left(
A^{l}\right)  }=\{%
{\displaystyle\sum\limits_{j=1}^{\infty}}
\lambda_{j}^{2l}(u,w_{j})\}^{\frac{1}{2}} \label{a5}%
\end{equation}
and for $u\in D(A^{l})$ we can write%
\[
A^{l}u=%
{\displaystyle\sum\limits_{j=1}^{\infty}}
\lambda_{j}^{l}(u,w_{j})w_{j}.
\]
Let us now define the trilinear form $b(.,.,.)$ associated with the inertia
terms%
\[
b\left(  u,v,w\right)  =\sum_{i,j=1}^{3}%
{\displaystyle\int\limits_{\Omega}}
u_{i}\frac{\partial v_{j}}{\partial x_{_{i}}}w_{j}dx.
\]
The continuity property of the trilinear form enables us to define (using
Riesz representation Theorem) a bilinear continuous operator $B\left(
u,v\right)  $; $V_{2}\times V_{2}\rightarrow V_{2}^{\prime}$ will be defined
by
\begin{equation}
\left\langle B\left(  u,v\right)  ,w\right\rangle =b\left(  u,v,w\right)
,\text{ }\forall w\in V_{2}\text{.} \label{2a}%
\end{equation}
Recall that for $u$ satisfying $\nabla.u=0$ we have%
\begin{equation}
b\left(  u,u,u\right)  =0\text{ and }b\left(  u,v,w\right)  =-b\left(
u,w,v\right)  \text{.} \label{2b}%
\end{equation}
We recall some well known inequalities that we will be using in what
follows.\newline Young's inequality%
\begin{equation}
ab\leq\frac{\sigma}{p}a^{p}+\frac{1}{q\sigma^{\frac{q}{p}}}b^{q}%
,a,b,\sigma>0,\text{ }p>1,\text{ }q=\frac{p}{p-1}. \label{7}%
\end{equation}
Poincar\'{e}'s inequality%
\begin{equation}
\lambda_{1}\left\Vert u\right\Vert ^{2}\leq\Vert A^{\frac{1}{2}}u\Vert
^{2}\text{\ for all }u\in V_{0}\text{.} \label{8}%
\end{equation}

Denoting%
\[
\Vert u\Vert_{G\left(  t\right)  }^{2}=\parallel e^{tA^{\frac{1}{2}}%
}u\parallel\text{\ and }\left(  u,v\right)  _{G\left(  t\right)
}=(e^{tA^{\frac{1}{2}}}u,e^{tA^{\frac{1}{2}}}v).
\]
The set $D(e^{\alpha A})$ is called the Gevrey class of operator of order
$\alpha\geq0$ \cite{11}. Our use of Gevrey classes shall be based on the
following consideration.

Denote with $N_{h}^{0}=\{x\in\Omega:h(x)=0\}$ the zero (nodal) set of a
function $h$ in a set $\Omega$, and let $\mathcal{H}^{2}$ be the
two-dimensional Hausdorff measure operating on subsets of $%
\mathbb{R}
^{3}$ (area in this case).

\section{Level Sets of the Vorticity Function}

Using the operators defined above, we can write the modified system (%
\ref{1}
) in the evolution form%

\begin{equation}%
\begin{array}
[c]{c}%
\dfrac{du}{dt}+\varepsilon A^{l}u+\nu Au+B\left(  u,u\right)  =f\left(
x\right)  \text{, in }\Omega\times\left(  0,\infty\right) \\
u_{0}\left(  x\right)  =u_{0},\text{ \ \ in }\Omega\text{.}%
\end{array}
\label{b1}%
\end{equation}
The existence and uniqueness results for initial value problem (%
\ref{1}
) can be found in \cite{18a}. The following theorem collects the main result
in this work

\begin{theorem}
\label{Theorem1} For $l\geq\frac{d+2}{4}$, $d$ is the space dimension, for
$\varepsilon>0$ fixed, $f\in\nolinebreak L^{2}\left(  0,T;V_{0}^{\prime
}\right)  $ and $u_{0}\in V_{0}$ be given. There exists a unique weak solution
of (%
\ref{1}
) which satisfies $u\in L^{2}\left(  0,T;V_{l}\right)  \cap L^{\infty}\left(
0,T;V_{0}\right)  ,\forall T>0.$
\end{theorem}

The modern understanding of turbulence is that it is a collection of weakly
correlated vortical motions, which, despite their intermittent and chaotic
distribution over a wide range of space and time scales, actually consist of
local characteristic 'eddy' patterns that persist as they move around under
the influences of their own and other eddies' vorticity fields \cite{13c}.

In fluid mechanics,the Reynolds number is important in analyzing any type of
flow when there is substantial velocity gradient (i.e. shear.) It indicates
the relative significance of the viscous effect compared to the inertia
effect. The Reynolds number is proportional to inertial force divided by
viscous force (see \cite{6} )%
\begin{equation}
Re=\frac{Ul}{\nu}\text{ \ \ \ \ \ \ \ }U^{2}=L^{-2}\left\langle \left\Vert
u\right\Vert _{2}^{2}\right\rangle \label{01}%
\end{equation}

where $l$ the characteristic scale of the forcing and $\left\langle
.\right\rangle $ is the long-time-average%
\begin{equation}
\left\langle g(.)\right\rangle =\lim_{T\rightarrow\infty}\sup(\frac{1}{T}%
{\textstyle\int_{0}^{T}}
g(t)dt). \label{02}%
\end{equation}

With Reynolds number calculator we can analyze what makes fluid flow regime
laminar and what is needed to force the fluid to flow in turbulent regime.
Experimental observations show that for 'fully developed' flow, laminar flow
occurs when $Re<$ $R_{_{l}}e$ and turbulent flow occurs when $Re>R_{_{t}}e$.
In the interval between $R_{_{l}}e$ and $R_{_{t}}e$, laminar and turbulent
flows are possible ('transition' flows) \cite{6} and references therein. The
nature of the vortex formed in the fluid flow depends strongly on the Reynolds
number(\cite{6}; and references therein). These transition Reynolds numbers
are also called critical Reynolds numbers, and were studied by Osborne
Reynolds around 1895 \cite{20b}. The transition to turbulence and the
constructon of vortex are delayed by increasing the critical Reynolds number.
If we assume that the critical Reynolds number $R_{_{c}}e$ for the onset of
vortex shedding is, atteint for
\begin{equation}
\left\Vert u\right\Vert =\frac{\nu R_{c}e}{l}L, \label{03}%
\end{equation}
then the associet velocity $u$ for each
\begin{equation}
Re\geq R_{c}e \label{04}%
\end{equation}
satisfaies the inequality
\begin{equation}
\left\Vert u\right\Vert \geq\frac{\nu R_{c}e}{l}L=\mu, \label{05}%
\end{equation}
$\mu$\ is a positive constant.

Another nondimensional quantity that we use often is the so-called Grashof
number, which is proportional to the forcing term $f$. Hence, we define the
Grashof numbers in the 3-dimensional case, as in Foias, Manley, Rosa and Temam
\cite{6} by%
\begin{equation}
Gr(f)=\frac{1}{\nu^{2}\lambda_{1}^{3/4}}\left\Vert f\right\Vert \label{06}%
\end{equation}
The effects of variation in Grashof number on vortex have been shown in the
work of Olson and Titi \cite{18c}, they keep the spatial structure of the
forcing function fixed and vary the Grashof number by varying the amplitude of
the forcing function. Namely, they vary the Grashof number by rescaling the
forcing function by a multiplicative factor. This is equivalent to changing
the viscosity or the size of the domain. As increases, or equivalently as the
viscosity decreases, the turbulent flow becomes more energetic and one would
expect the number of numerically determining modes to increase as well. There
are many reasons to suppose that the existene and intensty of vortex in our
work should increase as the grashof numbre increases \cite{13b,18c, 18b}. In
\cite{13b} zero forcing implies that the attractor has been reduced to zero.
Since all solutions decay eventually to zero in the unforced case.

This intuition is supported by existing theoretical critucal numbre
$G_{c}r(f)$ for the existence of level curves of representative vorticity fields.

Note the flow for $Gr(f)\geq G_{c}r(f)$ has noticeably more large scale
structure compared to the flow for $Gr(f)\leq G_{c}r(f)$. This is consistent
with the energy spectra, where most of the energy is in the lowest modes, that
is, in the large spatial scales and eddies when the Grashof number is large
\cite{18c}.

The effect of a body force on vorticity production and turbulence generation
in a fluid flow is described by the Grashof numbre.

In addition, we assume without loss of generality that $\left\Vert
f\right\Vert $ is bounded. Than, there exist a maximum Grashof numbre
$G_{\max}r(f)$ and a positive constant $\rho$ such that the body force $f$
satisfies the follwing inequality
\begin{equation}
\left\Vert f\right\Vert \leq\nu^{2}\lambda_{1}^{3/4}G_{\max}r(f)=\rho.
\label{07}%
\end{equation}
Since $\left\Vert f\right\Vert $ is srictement positive we get
\begin{equation}
\frac{\left\Vert u\right\Vert }{\left\Vert f\right\Vert }\geq\frac{\mu}{\rho
}=\frac{LR_{c}e}{\nu\lambda_{1}^{3/4}lG_{\max}r(f)}=\beta\label{08}%
\end{equation}
this gives a relation between $\left\Vert u\right\Vert $ and $\left\Vert
f\right\Vert $%
\begin{equation}
\left\Vert u\right\Vert \geq\beta\left\Vert f\right\Vert . \label{09}%
\end{equation}
Moreover, according to the definition of the Gevrey norm and the relation (%
\ref{09}
) we get%
\[
\left\Vert u\right\Vert _{G}\geq\beta\left\Vert f\right\Vert _{G}.
\]
The vorticity, $\omega=\nabla\times u$ satisfies the equation%
\begin{equation}
(\frac{d}{dt}+u.\nabla+\nu\triangle+\varepsilon\left(  -\triangle\right)
^{l})\omega=\omega.\nabla u+F \label{b2}%
\end{equation}
where $F=\nabla\times f.$

\begin{theorem}
\cite{16}Suppose that a nonzero function $h\in V_{1}$ satisfies%
\[
\left\Vert e^{\alpha A}h\right\Vert _{1}\leq M\left\Vert h\right\Vert _{1}%
\]
Then%
\[
\mathcal{H}^{2}\left(  N_{h}^{0}\right)  \leq C_{1}L^{2}\left(  1+\log
M\right)  e^{C_{2}L/\alpha}.
\]

\end{theorem}

Hereafter, $C_{i}$ for $i\in%
\mathbb{N}
$, stand for universal constants. The above statement is a special case of
\cite[Theorem 2.1]{16}. It will be used in conjunction with the following statement:

\begin{lemma}
\cite{15a}Let $u\in V_{0}$, and let $\omega$ and be its vorticity. If
\begin{equation}
\parallel A^{\frac{1}{2}}e^{\alpha A^{\frac{1}{2}}}u\parallel\leq M\parallel
A^{\frac{1}{2}}u\parallel\label{b8}%
\end{equation}
for some $M>0$, then, for every $c\in%
\mathbb{R}
$%
\begin{equation}
\parallel e^{\alpha A^{\frac{1}{2}}}\left(  \omega-c\right)  \parallel\leq
M\parallel\omega-c\parallel. \label{b9}%
\end{equation}

\end{lemma}

For the rest of the paper, let $u(t)$ be an arbitrary solution of the the
modified Navier Stokes system (%
\ref{1}
) with $u(0)=u_{0}$.

\begin{theorem}
Let $\left\Vert u\right\Vert \geq\beta\left\Vert f\right\Vert $ for any
$\alpha\leq\dfrac{\nu\lambda_{1}^{\frac{1}{2}}}{4}$ and $\beta\leq
\dfrac{4\sqrt{2}}{\nu}$. Then there exists a universal constant $C_{3}$ such
that if $\parallel A^{\frac{1}{2}}u\parallel\leq C_{3}$, then%
\begin{equation}
\parallel A^{\frac{1}{2}}e^{\alpha tA^{\frac{1}{2}}}u\parallel\leq2\parallel
A^{\frac{1}{2}}u_{0}\parallel,\text{ }t\geq0\text{.}\label{b3}%
\end{equation}

\end{theorem}

\begin{proof}
For any $\alpha$, $t\geq0$, We take the inner product of (%
\ref{b1}
) with $u$, to obtain%
\begin{equation}%
\begin{array}
[c]{ll}%
\dfrac{1}{2}\dfrac{d}{dt}\Vert A^{\frac{1}{2}}u\Vert_{G\left(  t\right)  }^{2}
& =\alpha\Vert A^{\frac{3}{4}}u\Vert_{G\left(  t\right)  }^{2}+(A\dot
{u},u)_{G\left(  t\right)  }\\
& =\alpha\Vert A^{\frac{3}{4}}u\Vert_{G\left(  t\right)  }^{2}-\varepsilon
\Vert A^{\frac{l+1}{2}}u\Vert_{G\left(  t\right)  }^{2}-\nu\Vert
Au\Vert_{G\left(  t\right)  }^{2}-b(u,u,Au)_{G\left(  t\right)  }%
+(f,Au)_{G\left(  t\right)  }.
\end{array}
\label{b4}%
\end{equation}
then using the Young's inequality (%
\ref{7}
) we have%
\[%
\begin{array}
[c]{ll}%
\dfrac{1}{2}\dfrac{d}{dt}\Vert A^{\frac{1}{2}}u\Vert_{G\left(  t\right)  }^{2}
& \leq-\varepsilon\Vert A^{\frac{l+1}{2}}u\Vert_{G\left(  t\right)  }%
^{2}+\dfrac{\nu}{4}\Vert Au\Vert_{G\left(  t\right)  }^{2}+\dfrac{\alpha^{2}%
}{\nu}\Vert A^{\frac{1}{2}}u\Vert_{G\left(  t\right)  }^{2}-\nu\Vert
Au\Vert_{G\left(  t\right)  }^{2}\\
& +(\dfrac{\nu}{2}\Vert Au\Vert_{G\left(  t\right)  }^{2}+\dfrac{1}{2\nu}\Vert
f\Vert_{G\left(  t\right)  }^{2})+b(u,u,Au)_{G\left(  t\right)  }%
\end{array}
\]
From we get%
\[
\Vert Au\Vert_{G\left(  t\right)  }^{2}\geq\Vert f\Vert_{G\left(  t\right)
}^{2}%
\]

this give
\[%
\begin{array}
[c]{ll}%
\dfrac{1}{2}\dfrac{d}{dt}\Vert A^{\frac{1}{2}}u\Vert_{G\left(  t\right)  }^{2}
& \leq-\varepsilon\Vert A^{\frac{l+1}{2}}u\Vert_{G\left(  t\right)  }%
^{2}+\dfrac{\nu}{4}\Vert Au\Vert_{G\left(  t\right)  }^{2}+\dfrac{\alpha^{2}%
}{\nu}\Vert A^{\frac{1}{2}}u\Vert_{G\left(  t\right)  }^{2}-\nu\Vert
Au\Vert_{G\left(  t\right)  }^{2}\\
& +(\dfrac{\nu}{2}\Vert Au\Vert_{G\left(  t\right)  }^{2}+\dfrac{_{\beta^{2}}%
}{2\lambda_{1}\nu}\Vert A^{\frac{1}{2}}u\Vert_{G\left(  t\right)  }%
^{2})+b(u,u,Au)_{G\left(  t\right)  }\\
& \leq-\varepsilon\Vert A^{\frac{l+1}{2}}u\Vert_{G\left(  t\right)  }%
^{2}+\dfrac{-\lambda_{1}\nu}{4}\Vert A^{\frac{1}{2}}u\Vert_{G\left(  t\right)
}^{2}+\dfrac{\alpha^{2}}{\nu}\Vert A^{\frac{1}{2}}u\Vert_{G\left(  t\right)
}^{2}\\
& +\dfrac{\beta^{2}}{2\lambda_{1}\nu}\Vert A^{\frac{1}{2}}u\Vert_{G\left(
t\right)  }^{2}+b(u,u,Au)_{G\left(  t\right)  }.
\end{array}
\]
We get for $\beta^{2}=\frac{1}{2\lambda_{1}\alpha^{2}}$%
\[%
\begin{array}
[c]{ll}%
\dfrac{1}{2}\dfrac{d}{dt}\Vert A^{\frac{1}{2}}u\Vert_{G\left(  t\right)  }%
^{2}\leq & -\varepsilon\Vert A^{\frac{l+1}{2}}u\Vert_{G\left(  t\right)  }%
^{2}-\dfrac{\lambda_{1}\nu}{4}\Vert A^{\frac{1}{2}}u\Vert_{G\left(  t\right)
}^{2}\\
& +\dfrac{2\alpha^{2}}{\nu}\Vert A^{\frac{1}{2}}u\Vert_{G\left(  t\right)
}^{2}+b(u,u,Au)_{G\left(  t\right)  }.
\end{array}
\]
We use the following inequality from \cite{11} and \cite[Section 4]{15a}%
\begin{equation}
b(u,u,Au)_{G\left(  t\right)  }\leq C_{4}\Vert A^{\frac{1}{2}}u\Vert_{G\left(
t\right)  }^{2}\Vert Au\Vert_{G\left(  t\right)  }(1+\log\dfrac{\Vert
Au\Vert_{_{G\left(  t\right)  }}^{2}}{\lambda_{1}\Vert A^{\frac{1}{2}}%
u\Vert_{_{G\left(  t\right)  }}^{2}})^{\frac{1}{2}} \label{b5}%
\end{equation}
to obtain%
\[%
\begin{array}
[c]{ll}%
\frac{1}{2}\dfrac{d}{dt}\Vert A^{\frac{1}{2}}u\Vert_{_{G\left(  t\right)  }%
}^{2} & \leq-\varepsilon\Vert A^{\frac{l+1}{2}}u\Vert_{_{G\left(  t\right)  }%
}^{2}-\dfrac{\lambda_{1}\nu}{4}\Vert A^{\frac{1}{2}}u\Vert_{_{G\left(
t\right)  }}^{2}+\dfrac{2\alpha^{2}}{\nu}\Vert A^{\frac{1}{2}}u\Vert
_{_{G\left(  t\right)  }}^{2}\\
& +C_{4}\Vert A^{\frac{1}{2}}u\Vert_{_{G\left(  t\right)  }}^{2}\Vert
Au\Vert_{_{G\left(  t\right)  }}(1+\log\dfrac{\Vert Au\Vert_{_{G\left(
t\right)  }}^{2}}{\lambda_{1}\Vert A^{\frac{1}{2}}u\Vert_{_{G\left(  t\right)
}}^{2}})^{\frac{1}{2}}.
\end{array}
\]
To establish (%
\ref{b3}
) we use the estimate \cite{15a}%
\begin{equation}
a\ \mu(1+\log\dfrac{\mu^{2}}{b^{2}})^{\frac{1}{2}}\leq d\mu^{2}+\dfrac{a^{2}%
}{d^{2}}\log\dfrac{2a}{bd}\text{ \ \ \ \ \ }a,d>0,\mu\geq b>0. \label{b6}%
\end{equation}
By applying the Poincar\'{e}'s inequality (%
\ref{8}
), we have that for $\mu=\Vert Au\Vert_{G\left(  t\right)  }$ and
$d=\dfrac{\nu}{8}$%
\[%
\begin{array}
[c]{ll}%
\dfrac{1}{2}\dfrac{d}{dt}\Vert A^{\frac{1}{2}}u\Vert_{G\left(  t\right)  }%
^{2}+\varepsilon\Vert A^{\frac{l+1}{2}}u\Vert_{G\left(  t\right)  }^{2}\leq &
-\frac{\lambda_{1}\nu}{8}\Vert A^{\frac{1}{2}}u\Vert_{G\left(  t\right)  }%
^{2}+\dfrac{2\alpha^{2}}{\nu}\Vert A^{\frac{1}{2}}u\Vert_{G\left(  t\right)
}^{2}\\
& +C_{5}\Vert A^{\frac{1}{2}}u\Vert_{G\left(  t\right)  }^{4}\log\frac
{C_{6}\Vert A^{\frac{1}{2}}u\Vert_{G\left(  t\right)  }}{\lambda_{1}^{\frac
{1}{2}}}.
\end{array}
\]
Letting $\alpha\leq\dfrac{\nu\lambda_{1}^{\frac{1}{2}}}{4}$ we have for
$\beta\leq\dfrac{4\sqrt{2}}{\nu}$ that%
\begin{equation}
\frac{1}{2}\dfrac{d}{dt}\Vert A^{\frac{1}{2}}u\Vert_{G\left(  t\right)  }%
^{2}\leq C_{5}\Vert A^{\frac{1}{2}}u\Vert_{G\left(  t\right)  }^{4}\log
\frac{C_{6}\Vert A^{\frac{1}{2}}u\Vert_{\Vert_{G\left(  t\right)  }^{2}}%
}{\lambda_{1}^{\frac{1}{2}}}. \label{b7}%
\end{equation}
If $\Vert A^{\frac{1}{2}}u_{0}\Vert<\frac{\lambda_{1}^{\frac{1}{2}}}{C_{6}%
}=C_{3}$, the term with a logarithm in (%
\ref{b7}
) is negative at $t=0$, and thus (%
\ref{b7}
) implies that $\Vert A^{\frac{1}{2}}u\Vert_{G\left(  t\right)  }$ is a
decreasing function of $t$.
\end{proof}

Theorem 3.4. implies that, for any solution $u(t)$ of (%
\ref{1}
), the space analyticity radius of $u(t)$ goes to infinity as $t\rightarrow
\infty$.

Let $\Omega$ be a periodic box, for simplicity assume $\Omega=(0,L)^{3}$, $A$
has eigenvalues $0<\lambda_{1}<\lambda_{2}<...$ with corresponding eigenspaces
$E_{1}$, $E_{2}$, ... Let $P_{m}$ be the projection on the eigenspaces
$E_{1}\oplus E_{2}\oplus...\oplus E_{m}$ and let $Q_{m}=I-P_{m}$ we have$\Vert
u\Vert^{2}=\Vert P_{m}u\Vert^{2}+\Vert Q_{m}u\Vert^{2}$ and we also have from
(%
\ref{a5}
) that%
\begin{equation}
\Vert A^{l}u\Vert\leq\lambda_{m}^{l}\Vert u\Vert\text{ for every }l\geq0\text{
and }u\in D(A^{l}). \label{c1}%
\end{equation}
For any $t\geq0$, let $\omega(t)$ be the vorticity of $u(t)$. We shall, for
any fixed $t>0$, estimate the quantity%
\begin{equation}
l\left(  \omega\left(  t\right)  \right)  =\sup_{c\in%
\mathbb{R}
}\mathcal{H}_{l}^{2}\left(  N_{\omega}^{c}\right)  . \label{c13}%
\end{equation}
Recall that for a function $h:\Omega\rightarrow%
\mathbb{R}
$, $N_{h}^{0}=\{x\in\Omega:h(x)=0\}$. We need the following fact

\begin{lemma}
Let $\left\Vert u\right\Vert \geq\beta\left\Vert f\right\Vert $ for any
$\beta\geq0$. Then%
\begin{equation}
\Vert u\left(  t\right)  \Vert\geq\Vert u(0)\Vert\exp(\eta t)\text{ for every
}t\geq0\text{.} \label{c2}%
\end{equation}
With $\eta=-(\varepsilon\lambda_{m}^{\frac{l}{2}}+\frac{1+\beta^{2}}%
{2\beta^{2}})$.
\end{lemma}

\begin{proof}
Taking the scalar product of both sides of (%
\ref{1}
) by $u(t)$ and using (%
\ref{2b}
), we have that%
\begin{equation}
\frac{1}{2}\dfrac{d}{dt}\Vert u\Vert^{2}+\nu\Vert A^{\frac{1}{2}}u\Vert
^{2}+\varepsilon\Vert A^{\frac{l}{2}}u\Vert^{2}=\left(  f,u\right)  \text{ for
}t\geq0. \label{c10}%
\end{equation}
Using (%
\ref{c1}
) and the following inequality%
\begin{equation}
\left(  f,u\right)  \geq-\dfrac{1}{2}\Vert f\Vert^{2}-\frac{1}{2}\Vert
u\Vert^{2} \label{c11}%
\end{equation}
Because the increasing sequence $0\leq\lambda_{1}\leq\lambda\leq\lambda_{m}$
we have%
\begin{equation}
\dfrac{1}{2}\dfrac{d}{dt}\Vert u\Vert^{2}\geq(-\nu\lambda\Vert A^{\frac{1}{2}%
}u\Vert^{2}-(\varepsilon\lambda_{m}^{\frac{l}{2}}+\dfrac{1+\beta^{2}}%
{2\beta^{2}})\Vert u\Vert^{2} \label{c12}%
\end{equation}
note that since%
\[
\parallel A^{\frac{1}{2}}u\parallel\leq C_{3}%
\]
we have that%
\begin{equation}
\dfrac{1}{2}\dfrac{d}{dt}\Vert u\Vert^{2}\geq-C_{3}-(\varepsilon\lambda
_{m}^{\frac{l}{2}}+\dfrac{1+\beta^{2}}{2\beta^{2}})\Vert u\Vert\label{c17}%
\end{equation}
if we set $\eta=-(\varepsilon\lambda_{m}^{\frac{l}{2}}+\dfrac{1+\beta^{2}%
}{2\beta^{2}})$ then we have from (%
\ref{c17}
) that%
\[
\dfrac{1}{2}\dfrac{d}{dt}\Vert u\Vert^{2}\geq-C_{3}+\eta\Vert u\Vert.
\]
Integrating the above inequality from $0$ to $t$, we get%
\begin{equation}
\Vert u\Vert^{2}\geq\frac{-C_{3}}{\eta}\left(  1-\exp(\eta t)\right)  +\Vert
u(0)\Vert^{2}\exp(\eta t) \label{c18}%
\end{equation}
or, since%
\[
\frac{-C_{3}}{\eta}\left(  1-\exp(\eta t)\right)  \geq0.
\]
Thus, we have the inequality (%
\ref{c2}
).
\end{proof}

\begin{proposition}
Let $\left\Vert u\right\Vert \geq\beta\left\Vert f\right\Vert $ and $u_{0}%
\neq0$, and suppose that $\left\Vert u_{0}\right\Vert \leq C_{3}\nu\lambda
_{1}^{\frac{1}{2}}$.Then we have that%
\begin{equation}
l\left(  \omega\left(  t\right)  \right)  \leq C_{1}L(1+\frac{1}{2}%
Log\dfrac{\lambda_{m}}{\lambda_{1}}+(\varepsilon\lambda_{m}^{\frac{l}{2}%
}+\frac{1+\beta^{2}}{2\beta^{2}})t)e^{\frac{C_{2}L}{\alpha t}}\text{
for\ }t\geq0, \label{c14}%
\end{equation}
for any $\alpha\leq\dfrac{\nu\lambda_{1}^{\frac{1}{2}}}{4}$ and $\beta
\leq\dfrac{4\sqrt{2}}{\nu}$.
\end{proposition}

\begin{proof}
By Theorem 3.4 and Lemma 3.5, we get for $t\geq0$ the following%
\begin{equation}%
\begin{array}
[c]{ll}%
\parallel A^{\frac{1}{2}}e^{\alpha tA^{\frac{1}{2}}}u\parallel &
\leq2\parallel A^{\frac{1}{2}}u_{0}\parallel\\
& \leq2\lambda^{\frac{1}{2}}\left\Vert u_{0}\right\Vert
\end{array}
\label{c15}%
\end{equation}
and use the inequality (%
\ref{c2}
)\ to get%
\begin{equation}%
\begin{array}
[c]{ll}%
\parallel A^{\frac{1}{2}}e^{\alpha tA^{\frac{1}{2}}}u\parallel & \leq
2\lambda^{\frac{1}{2}}\left\Vert u\left(  t\right)  \right\Vert \exp
(\varepsilon\lambda_{m}^{\frac{l}{2}}+\frac{1+\beta^{2}}{2\beta^{2}})t\\
& \leq2(\frac{\lambda}{\lambda_{1}})^{\frac{1}{2}}\parallel A^{\frac{1}{2}%
}u\left(  t\right)  \parallel\exp(\varepsilon\lambda_{m}^{\frac{l}{2}}%
+\frac{1+\beta^{2}}{2\beta^{2}})t.
\end{array}
\label{c16}%
\end{equation}
The rest follows by combining (%
\ref{c16}
) with Lemma 3.3 and Theorem 3.2.
\end{proof}

The foundational result for our two-dimensional Hausdorff measure estimates of
$N_{\omega}^{0}=\{x\in\Omega\subset%
\mathbb{R}
^{3}/$ $\omega(x,t)=0\}$ the level sets of the vorticity $\omega$\ of
solutions is

\begin{theorem}
Let $\left\Vert u\right\Vert \geq\beta\left\Vert f\right\Vert $ and $u_{0}%
\neq0$, and suppose that $\left\Vert u_{0}\right\Vert \leq C_{3}\nu\lambda
_{1}^{\frac{1}{2}}$. Then%
\begin{equation}
l\left(  \omega\left(  t\right)  \right)  \leq C_{7}\lambda_{m}^{\frac{l}{2}%
}\text{\ for\ }t\geq t_{0}, \label{c4}%
\end{equation}
with $t_{0}=\frac{2C_{2}L}{\nu\lambda_{1}^{\frac{1}{2}}}$ for any $\alpha
\leq\dfrac{\nu\lambda_{1}^{\frac{1}{2}}}{4}$ and $\beta\leq\dfrac{4\sqrt{2}%
}{\nu}$.
\end{theorem}

\begin{proof}
With $t\geq\frac{2C_{2}L}{\nu\lambda_{1}^{\frac{1}{2}}}$ the inequality (%
\ref{c14}
) implies%
\begin{equation}
l\left(  \omega\left(  t\right)  \right)  \leq C_{1}e(2+\dfrac{1}{2}%
Log\dfrac{\lambda_{m}}{\lambda_{1}}+(\varepsilon\lambda_{m}^{\frac{l}{2}%
}+\frac{1+\beta^{2}}{2\beta^{2}})\frac{C_{2}L}{\alpha})\text{\ for\ }t\geq0.
\label{c5}%
\end{equation}
Since $\lambda_{m}\geq\lambda_{1}$ (%
\ref{c5}
) follows directly from the above inequality.
\end{proof}

The estimate of the Hausdorff measure $\mathcal{H}_{l}^{2}$ grows in $m$ due
to the term $\frac{\lambda_{m}}{\lambda_{1}}$ but at a rate lower than
$\frac{l}{3}$.

\begin{proposition}
Let $\left\Vert u\right\Vert \geq\beta\left\Vert f\right\Vert $ and $u_{0}%
\neq0$, and suppose that $\left\Vert u_{0}\right\Vert \leq C_{3}\nu\lambda
_{1}^{\frac{1}{2}}$. Then%
\begin{equation}
\sup_{t\rightarrow\infty}l\left(  \omega\left(  t\right)  \right)  \leq
C_{8}m^{\frac{l}{3}}\text{\ for\ }t\geq t_{0}. \label{c6}%
\end{equation}
for any $\alpha\leq\dfrac{\nu\lambda_{1}^{\frac{1}{2}}}{4}$ and $\beta
\leq\dfrac{4\sqrt{2}}{\nu}$.
\end{proposition}

\begin{proof}
Note that in the 3D case we have $\lambda_{j}\geq C_{9}L^{-2}j^{\frac{2}{3}}$
for some positive universal constant (see, for example \cite[Lemma VI 2.1]%
{21}). Therefore, Since $\lambda_{m}\backsim\lambda_{1}m^{\frac{2}{3}}$ the
growth in $m$ of the Hausdorff measure (%
\ref{c6}
) is less than $m^{\frac{l}{3}}$.
\end{proof}

If we impose the condition $\lambda_{m}\leq(\frac{1}{l_{\epsilon}})^{2}$ or
$\frac{\lambda_{m}}{\lambda_{1}}\leq(\frac{l_{0}}{l_{\epsilon}})^{2}$ where
$l_{0}=\lambda_{1}^{\frac{-1}{2}}$ represents characteristic macroscopic
length, and $l_{\epsilon}$ is the Kolmogorov length scale, i.e. $l_{\epsilon
}=\nolinebreak\frac{\nu^{3}}{\epsilon}$ where $\epsilon$ is\ Kolmogorov's mean
rate of dissipation of energy in turbulent flow (see e.g. \cite{0b, 6, 13a,
21}, and the references contained therein) is defined as
\[
\epsilon=\lambda_{1}^{\frac{3}{2}}\nu\lim\sup_{T\rightarrow\infty}%
\int\limits_{0}^{T}\parallel A^{\frac{l}{2}}\parallel_{2}^{2}ds.
\]
Substituting this in (%
\ref{c4}
) gives%
\begin{equation}
\sup_{t\rightarrow\infty}l\left(  \omega\left(  t\right)  \right)  \leq
C_{10}(\frac{l_{0}}{l_{\epsilon}})^{\frac{l}{3}}. \label{c8}%
\end{equation}
Since the (dimensionless) Grashoff number $G=\frac{\sup_{t\geq0}\left\Vert
f\right\Vert _{2}^{2}}{\nu^{3}\lambda_{1}^{\frac{3}{2}}}$ in 3D (see e.g.
\cite{0b, 6, 21}) is an upper bound for $(\frac{l_{0}}{l_{\epsilon}})^{2}$.
Hence, we obtain for the Hausdorff measure of the equation (%
\ref{1}
) the following estimate in terms of the Grashoff number $G$.

\begin{proposition}
Let $\left\Vert u\right\Vert \geq\beta\left\Vert f\right\Vert $ and $u_{0}%
\neq0$, and suppose that $\left\Vert u_{0}\right\Vert \leq C_{3}\nu\lambda
_{1}^{\frac{1}{2}}$. Then for any $\alpha\leq\dfrac{\nu\lambda_{1}^{\frac
{1}{2}}}{4}$ and $\beta\leq\dfrac{\nu\lambda_{1}}{2\sqrt{2}}$ we have%
\begin{equation}
\sup_{t\rightarrow\infty}l\left(  \omega\left(  t\right)  \right)  \leq
C_{11}G^{^{\frac{l}{6}}}\text{ for\ }t\geq t_{0}. \label{c9}%
\end{equation}

\end{proposition}

This result holds independently of $m$, with $C_{11}$ independent of $m$. The
estimate grows in $l$ at a rate lower than $\frac{l}{6}$. If we impose the
condition $l=1$, the estimates become $\sup_{t\rightarrow\infty}l\left(
\omega\left(  t\right)  \right)  \leq C_{11}G^{^{\frac{1}{6}}}$. This result
recover the usual 3D Navier--Stokes equations estimates, for the
two-dimensional Hausdorff measure $\mathcal{H}_{1}^{2}(N_{\omega}^{0})$
estimates of the level sets associated with the vorticity. Here again our
results indicate that under certain conditions the upper bounds for
$\mathcal{H}_{1}^{2}(N_{\omega}^{0})$ converge to the associated upper bounds
of the two-dimensional Hausdorff measure $\mathcal{H}_{1}^{2}(N_{\omega}^{0})$
estimates for the usual 3D Navier--Stokes equations with potential force.

\section{\textbf{Conclusion}}

Proving global regularity for the 3D Navier--Stokes equations is one of the
most challenging outstanding problems in nonlinear analysis. The main
difficulty in establishing this result lies in controlling certain norms of
vorticity. More specifically, the vorticity stretching term in the 3D
vorticity equation forms the main obstacle to achieving this control, C. Foias
\cite{5a} and estimates on the number of degrees of freedom for the
Navier-Stokes equations and its closure models are a measure of the complexity
of the system J. Avrin \cite{0b}. This paper proposed another interesting way
to estimate this complexity through bounding the size of the nodal set for the
vorticity and expressing this estimate in terms of $G$.

We provide upper bounds for the two-dimensional Hausdorff measure
$\mathcal{H}_{l}^{2}$ of the level sets associated with the vorticity of
modified three dimensional Navier-Stokes equations$\ $this bounds depend on
$m$ and $l$,$\ $this dependence is a fractional power of $l$. Thus answering a
question raised by J. Avrin \cite{0b}.

Another interesting way to study decaying turbulence in the three-dimensional
incompressible Navier--Stokes equations is to prvide a numerical investigation
of our theoretical results on the size of the nodal set for the vorticity in
the dependence of turbulence structure and vortex dynamics, as was done in
\cite{18c} for the number of numerically determining modes in the 2D
Navier--Stokes equations. It would be interesting to see how the turbulence
structure depend on $l$.

\end{document}